\documentclass[11]{amsart}
\usepackage{fullpage}
\usepackage{amssymb}
\usepackage{amsmath}
\usepackage{amsfonts}
\usepackage{subfig}
\usepackage{amsthm}
\usepackage{enumerate}
\usepackage{graphicx}
\usepackage{enumitem}
\usepackage{color}
\usepackage{url}
\usepackage{pgf}
\usepackage{tikz}
\usepackage{pgfplots}
\usepackage{mathrsfs}
\usetikzlibrary{arrows}
\usepackage{tikz-cd}
\usepackage{MnSymbol}
\usetikzlibrary{shapes.misc}

\tikzset{cross/.style={cross out, draw=black, minimum size=2*(#1-\pgflinewidth), inner sep=0pt, outer sep=0pt},
cross/.default={1pt}}

\theoremstyle{plain}
\newtheorem{theorem}{Theorem}
\newtheorem{lemma}[theorem]{Lemma}
\newtheorem{corollary}[theorem]{Corollary}

\newtheorem{remark}[theorem]{Remark}
\newtheorem{definition}[theorem]{Definition}

\theoremstyle{aim}

\newtheorem{conjecture}[theorem]{Conjecture}
\theoremstyle{conjecture}

\numberwithin{equation}{section}

\newcommand{\cP}{{\mathcal{P}}}
\newcommand{\F}{{\mathbb{F}_q}}
\newcommand{\PG}{\mathrm{PG}(3,q)}
\newcommand{\PGG}{\mathrm{PG}}
\newcommand{\C}{\mathcal{C}}

\newcommand{\Ima}{Im}

\newcommand{\PGL}{\mathrm{PGL}}

\newcommand{\Cu}{\mathbf{C}}
\newcommand{\cZ}{\mathcal{Z}}

\newcommand{\cO}{\mathcal{O}}
\newcommand{\cH}{\mathcal{H}}
\newcommand{\cL}{\mathcal{L}}
\newcommand{\bF}{\mathbb{F}}

\begin{document}
\title{On pencils of cubics on the projective line over finite fields of characteristic $>3$}

\author{G\"{u}lizar G\"{u}nay}
\address{G\"{u}lizar G\"{u}nay, Faculty of Engineering and Natural Sciences, Sabanc{\i} University, Tuzla, Istanbul 34956, Turkey}
\email{gunaygulizar@sabanciuniv.edu}

\author{Michel Lavrauw}
\address{Michel Lavrauw, Faculty of Engineering and Natural Sciences, Sabanc{\i} University, Tuzla, Istanbul 34956, Turkey}
\email{mlavrauw@sabanciuniv.edu}
\thanks{The authors were supported by {\em The Scientific and Technological Research Council of Turkey}, T\"UB\.{I}TAK (project no. 118F159)}
\begin{abstract}
In this paper we study combinatorial invariants of the equivalence classes of pencils of cubics on $\PGG(1,q)$, for $q$ odd and $q$ not divisible by 3. These equivalence classes are considered as orbits of lines in $\PG$, under the action of the subgroup $G\cong \PGL(2,q)$ of $\PGL(4,q)$ which preserves the twisted cubic $\C$ in $\PG$. In particular we determine the point orbit distributions and plane orbit distributions of all $G$-orbits of lines which are contained in an osculating plane of $\C$, have non-empty intersection with $\C$, or are imaginary chords or imaginary axes of $\C$.
\end{abstract}
\maketitle

\section{Introduction and motivation}
\bigskip Let $\PGG(n,q)$ denote the $n$-dimensional projective space over a finite field $\F$. We assume that $q$ is an odd prime power and not divisible by $3$. The projective linear group, denoted by $\PGL(n+1,q)$, is defined as the group of {\it projectivities} induced by the action of the general linear group ${\mathrm{GL}}(n+1,q)$ on $\PGG(n,q)$.

The cubic forms on $\PGG(1,q)$ form a four-dimensional vector space $W$, and subspaces of the projective space $\PGG(W)$ are called {\it linear systems of cubics}. One-dimensional linear systems are called \emph{pencils}. In this paper we are interested in the equivalence classes of pencils of cubics in $\PGG(1,q)$ under the action of $\PGL(2,q)$.

Pencils of cubics can also be seen as vectors in the space of partially symmetric tensors $S^3\bF_q^2\otimes \bF_q^2$, where $S^3\bF_q^2$ denotes the space of symmetric tensors in $\bF_q^2\otimes \bF_q^2\otimes \bF_q^2$. As such, our results fit into the larger research project of classifying tensors over finite fields \cite{talk}, and our approach is inspired by the classifications obtained in \cite{LaSh2015}, \cite{LaPo2020} and \cite{LaPoSh2020}. The combinatorial invariants which are studied here, were motivated by the combinatorial invariants for the orbits of nets of conics in studied in \cite{LaPoSh2021}.

In $\PGG(n,q)$, for $2\leq n \leq q-2$, a {\it normal rational curve} (NRC) is an algebraic curve projectively equivalent to \[\mathcal{K}=\{(1,t,t^2,...,t^n):t\in \F\}\cup\{(0,0,...,1)\}.\]
A normal rational curve in $\PG$ is called a \emph{twisted cubic}. The twisted cubic is left invariant under the action of a group $G\leq \PGL(4,q)$ isomorphic to $\PGL(2,q)$ (see e.g. Harris \cite[p. 118]{Harris1992}).

The classification of $G$-orbits of lines in $\PG$ is equivalent to the classification of pencils of cubics in $\PGG(1,q)$.

The $G$-orbits on points and planes of $\PG$ are well understood (see e.g. \cite{MR500485}). In this paper we study ten of the $G$-orbits on lines in $\PG$ and determine the combinatorial invariants called the point orbit distributions and plane orbit distributions.

\bigskip The twisted cubic in $\PG$ has some remarkable geometric properties which have led to many interesting applications.
It is well known that the twisted cubic is the intersection of three quadrics (see e.g. Harris \cite{Harris1992}). Also, the points on a twisted cubic form an arc (i.e. no four points are coplanar) of size $q+1$, and the related MDS code is known as the (extended) Reed-Solomon code; it has length $q+1$, dimension $4$ and minimum distance $q-2$. We refer to \cite{BaLa2019} for a recent survey on arcs (including proofs of several instances of the MDS conjecture).

\bigskip As further motivation for our work, we mention some results in which the twisted cubic plays a fundamental role. In \cite{MR500485} Bruen and Hirschfeld used the property that the twisted cubic is an arc to construct spreads of $\PG$ which do not contain any regulus. This leads to interesting examples of non-Desarguesian projective planes using the Andr\'e-Bruck-Bose construction of translation planes from spreads. Furthermore, Bonoli and Polverino \cite{BoPo2005} used the twisted cubic to prove the non-existence of certain spreads in the classical generalized hexagon $H(q^n)$, provided that $q$ and $n$ satisfy the bound from \cite{BaBlLa2003}. More recently, Bartoli et al. \cite{BaDaMaPa2020} studied certain properties of the twisted cubic in relation to optimal multiple covering codes.

\bigskip

The paper is organized as follows. In Section 2 we introduce the notation and terminology, and collect some of the results concerning the twisted cubic in $\PG$ which will be used throughout the paper.
In Section 3 we determine the point orbit distributions and the plane orbit distributions of the lines contained in osculating planes. This comprises most of the computations in this paper, and concerns the line orbits $\cL_i$, $i\in \{1,\ldots,5\}$. Section 4 contains the orbit distributions of the lines meeting the twisted cubic (orbits $\cL_6$, $\cL_7$, and $\cL_8$). These easily follow from the results of Section 3. In Section 5 the orbit distributions of the imaginary chords and imaginary axes are determined. These are the orbits $\cL_9$ and $\cL_{10}$.

%

\section{Preliminaries}

In this section we review some theory and definitions used in our study. They are well known and can be extracted from standard textbooks on finite projective geometry.

\bigskip Given the homogeneous polynomials $f_1,\ldots,f_t \in \F[X_0,X_1,..,X_n]$, we denote by ${\mathcal{Z}}(f_1,\ldots,f_t)$ the projective algebraic variety they define. Given an algebraic variety $\mathcal X$ defined (by homogeneous polynomials) over $\bF_q$, the set of $\mathbb{F}_{q^n}$-rational points of $\mathcal X$ in $\PGG(n,q^n)$ is denoted by ${\mathcal{X}}(\mathbb{F}_{q^n})$. For reasons of simplicity, these notations will sometimes be abused, for example we also write $\mathcal X$ instead of ${\mathcal{X}}(\bF_q)$, but it will be clear from the context what is meant.
A \emph{cubic} $\Cu$ in $\PGG(1,q)$ is the zero locus of a homogenous polynomial $f(X_0,X_1)$ of degree $3$ in $\F[X_0,X_1]$ (called a {\it cubic form on $\PGG(1,q)$}), in short $\Cu=\mathcal{Z}(f)$.

\medskip

The pencil defined by two cubics $\Cu_1=\mathcal{Z}(f_1)$ and $\Cu_2=\mathcal{Z}(f_2)$ is \[\mathcal{P}(\Cu_1,\Cu_2)=\{\mathcal{Z}(\alpha f_1+\beta f_2)\;|\;(\alpha,\beta)\in \PGG(1,q)\}.\] The intersection $\Cu_1\cap \Cu_2$ is called a \emph{base\;locus} and its points are called the \emph{base\;points} of the pencil.

The veronese map $\nu_3$ from $\PGG(1,q)$ to $\PGG(3,q)$ is defined by
 \[\nu_3: (x_0,x_1) \; \mapsto \; (x_0^3,x_0^2x_1,x_0x_1^2,x_1^3)\]
 and its image is the twisted cubic $\mathcal{C}$.
Consider also the map $\delta_3$ from the set of cubics of $\PGG(1,q)$ to the set of planes of $\PG$ where a cubic  $\Cu=\mathcal{Z}(f)$ with \[f= y_{30} X_0^3+y_{21} X_0^2X_1+y_{12} X_0X_1^2+ y_{03} X_1^3 \in \F[X_0,X_1]\] in $\PGG(1,q)$ is mapped by $\delta_3$ onto the plane with dual coordinates $[y_{30},y_{21},y_{12},y_{03}]$.

The image of a cubic $\Cu$ of $\PGG(1,q)$ under $\delta_3$ is a plane $\delta_3(\Cu)$ in  $\PGG(3,q)$ and the set of rational points of the cubic $\Cu$ corresponds to the set of points of the twisted cubic $\mathcal C$ which are contained in $\delta_3(\Cu)$.
Conversely, every cubic in $\PGG(1,q)$ is the preimage of a plane in $\PGG(3,q)$. To avoid confusion, the indeterminates of the coordinate ring of $\PGG(1,q)$ are denoted by $X_0,X_1$, while the indeterminates of the coordinate ring of $\PGG(3,q)$ are denoted by $Y_0,Y_1,Y_2,Y_3$.


For each $\alpha\in \mathrm{PGL}(2,q)$, there exists an $\tilde{\alpha}\in \PGL(4,q)$ such that $(\nu_3(P))^{\tilde{\alpha}}=\nu_3(P^{\alpha})$ for every $P\in \PGG(1,q)$. Explicitly, consider the homomorphism
\begin{eqnarray}\label{varphi}
   \varphi:\PGL(2,q) \rightarrow\PGL(4,q):\alpha \mapsto \tilde{\alpha}
\end{eqnarray}
where the image under $\varphi$ of a projectivity $\alpha$ induced by
$$\left(
                                                           \begin{array}{cc}
                                                             a & b \\
                                                             c & d \\
                                                           \end{array}
                                                         \right)
$$
is mapped to the projectivity induced by
\[\left(
                                                 \begin{array}{cccc}
                                                 a^3 &a^2b &ab^2&b^3 \\
                                                 3a^2c & a^2d+2abc & b^2c+2abd&3b^2d \\
                                                  3ac^2 & bc^2+2acd & ad^2+2bcd&3bd^2 \\
                                                   c^3 & c^2d & cd^2&d^3 \\
                                                   \end{array}
                                               \right).\]
Denote the image of $\varphi$ by $G$. Then $G\cong \PGL(2,q)$ and $G$ acts $3$-transitively on $\mathcal{C}$.\\



A line $\ell$ of $\PG$ is a \emph{chord} of $\mathcal{C}$ if it is a line joining  either two distinct  points of $\mathcal{C}$, two coinciding points, or a pair of conjugate points $(P,P^q)$ of $\mathcal{C}({\mathbb{F}}_{q^2})\setminus \mathcal{C}({\mathbb{F}}_{q})$. The line $\ell$ is, respectively, called a \emph{real chord}, a \emph{tangent} or an \emph{imaginary chord} of $\mathcal{C}$. Let $P(t)$ denote the point with coordinates $(1,t,t^2,t^3)$ of ${\mathcal{C}}({\mathbb{F}}_{q^n})$, for $t\in {\mathbb{F}}_{q^n}$, and $P(\infty)$ the point of $\mathcal C$ with coordinates $(0,0,0,1)$.
The chord $\ell=\langle P(t_1), P(t_2)\rangle$, $t_1,t_2\notin\{0,\infty\}$, is determined by the equations
\[ \left\{
            \begin{array}{ll}
             t_1t_2Y_1-(t_1+t_2)Y_2+Y_3=0, & \\
            -t_1t_2Y_0+(t_1+t_2)Y_1-Y_2=0. &
            \end{array}\right.
          \]
Moreover, the tangent line $\ell$ of $\C$ at the point $P(t)$, $t\notin\{0,\infty\}$, is determined by the equations
\[\Bigg\{\begin{array}{cccc}
                                                                                                                t^2Y_1-2tY_2+Y_3=0, \\
                                                                                                               -t^2Y_0+2tY_1-Y_2=0,
                                                                                                             \end{array}\]

while the tangent line of $\C$ at the point $P(0)$ is $ \mathcal{Z}(Y_2,Y_3)$ and the tangent line of $\C$ at $P(\infty)$ is $\mathcal{Z}(Y_0,Y_1)$.
A plane $\Pi$ intersecting the twisted cubic in a triple point $P$ is called the \emph{osculating plane} at $P$.
The osculating plane at $P(t)$ is denoted by $\Pi(t)$ and has equation
 \[\Pi(t): -t^3Y_0+3t^2 Y_1 -3t Y_2 +Y_3=0.\]

Consider a point $P_1(y_0,y_1,y_2,y_3)$ of $\PG$. 
Through $P_1$, there are $3$ osculating planes $\Pi(t_1)$, $\Pi(t_2)$ and $\Pi(t_3)$ of $\C$ with  \[{[}-y_0:3y_1:-3y_2:y_3{]}={[}1:-(t_1+t_2+t_3):(t_1t_2+t_1t_3+t_2t_3): -t_1t_2t_3{]}.\]
The plane $\Pi'$ joining the points $P(t_1)$, $P(t_2)$, $P(t_3)$ is
\[\Pi':y_3Y_0-3y_2Y_1+3y_1Y_2-y_0Y_3=0.\]

The map $\sigma$ mapping the point with coordinates $(y_0,y_1,y_2,y_3)$ to the plane with dual coordinates $[-y_3,3y_2$, $-3y_1,y_0]$
is a symplectic polarity of $\PG$, and the plane $P^{\sigma}$ is called the \emph{polar\;plane} of the point $P$.
Note that for a point $P(t)$ of $\C$, the polar plane of $P(t)$ is the osculating plane $\Pi(t)$ of $\C$ at $P(t)$.
The polar plane of a point $P$ not on $\C$ is the span of the contact points of the three osculating planes through $P$.

The osculating planes of $\C$ form the \emph{osculating developable} $\Gamma$ to $\C$. The symplectic polarity $\sigma$ sends the chords of the twisted cubic $\C$ to the \emph{axes} of $\Gamma$.
\par A point $P$ lies on a line $\ell$ in $\PG$ if and only if $\ell^\sigma$ is contained in the plane $P^\sigma$ where $\sigma$ is the polarity in $\PG$ defined as above. For a line $\ell$, if $\ell^{\sigma}=\ell$, then $\ell$ is called \emph{self\;polar}. We will often refer to the following well known lemma.

\begin{lemma}\label{selfdual}
The self polar lines which pass through a given point $P$ of $\C$ are all the lines through $P$ in the polar plane $P^{\sigma}$ of $P$.
\end{lemma}
For a point $P\in \C$, there is a unique tangent line $\ell$ through $P$ in the osculating plane $P^{\sigma}$. By Lemma~\ref{selfdual}, the tangent line $\ell$ is a self polar line through $P$. Also, a unisecant line in the osculating plane $P^{\sigma}$ which passes through the point $P \in \C$ is a self polar line. The following lemma describes the polar of the intersection of two osculating planes.

\begin{lemma} \label{chord} If $\ell$ is a line contained in the intersection of two osculating planes $\Pi_1$, $\Pi_2$ where $\Pi_1$ is an osculating plane at $P_1$ of $\C$ and $\Pi_2$ is an osculating plane at $P_2$ of $\C$, then $\ell^{\sigma}$ is a chord through points $P_1$, $P_2$ of the twisted cubic $\C$.
\end{lemma}

\par The lines of $\PG$ can be partitioned into $6$ classes which are unions of orbits under $G$. Each of these classes is denoted by $\mathcal{O}_i$, $i\in \{1,...,6\}$, and the set of polar lines of lines in the class $\mathcal{O}_i$ is denoted by $\mathcal{O}^\perp_i$.
\begin{itemize}
\item The class $\mathcal{O}_1$ contains real chords of the twisted cubic $\C$ and has size $q(q+1)/2$. Its dual $\mathcal{O}^\perp_1$ contains the real axes of $\Gamma$.
\item The class $\mathcal{O}_2=\mathcal{O}^\perp_2$ contains tangents of the twisted cubic $\C$ and has size $q+1$.
\item The class $\mathcal{O}_3$ contains imaginary chords and has size $q(q-1)/2$. Its dual $\mathcal{O}^\perp_3$ contains the imaginary axes of $\Gamma$.
\item The class $\mathcal{O}_4=\mathcal{O}^\perp_4$ contains non-tangent unisecants  in osculating planes and has size $q(q+1)$.
\item The class $\mathcal{O}_5$ contains unisecants not in osculating planes and has size $q(q^2-1)$. Its dual $\mathcal{O}^\perp_5$ contains external lines in osculating planes.
\item The class $\mathcal{O}_6=\mathcal{O}^\perp_6$ contains external lines, (not chords and not in osculating planes) and has size $q(q-1)(q^2-1)$.
\end{itemize}

\par As mentioned above, the $G$-orbits on points and planes of $\PGG(3,q)$ are well understood. In order to define the combinatorial invariants for the $G$-orbits on lines (which will be introduced below) which form the main body of our study, we give the description of these $G$-orbits. Also, since the classification of lines in $\PGG(3,q)$ under the action of $G$ is equivalent to the classification of pencils of cubics in $\PGG(1,q)$, we include the relation between cubics of $\PGG(1,q)$ and planes of $\PG$ (under the map $\delta_3$).
There are $5$ types of cubics in $\PGG(1,q)$, giving $5$ $G$-orbits of planes in $\PGG(3,q)$.
\begin{itemize}
\item The cubics which consist of a triple point are called \emph{type\;$1$} cubic. The corresponding planes are the osculating planes and they form one $G$-orbit $\mathcal{H}_1$ of size $q+1$.
\item A cubic which has a double point and a single point is called a \emph{type\;$2$} cubic. The corresponding planes form one $G$-orbit. There are $q(q+1)$ such planes, and this orbit is denoted by $\mathcal{H}_2$.
\item A cubic which has $3$ distinct points is called a \emph{type\;$3$} cubic, and corresponds to a plane meeting the twisted cubic in 3 distinct points. There are $q(q^2-1)/6$ such planes and they form one $G$-orbit denoted by $\mathcal{H}_3$.
\item A cubic which has a $1$ single point and $2$ imaginary points is called a \emph{type\;$4$} cubic.  These cubics correspond to the $q(q^2-1)/2$ planes through exactly $1$ point of the twisted cubic $\C$, which are not osculating planes. They form the $G$-orbit $\cH_4$.
\item A cubic which has $3$ imaginary points is called a \emph{type\;$5$} cubic.
The corresponding $G$-orbits of planes is denoted by $\cH_5$ and consists of $q(q^2-1)/3$ planes.
\end{itemize}

The five types of cubics on $\PGG(1,q)$ (equivalently the five $G$-orbits $\cH_1, \ldots, \cH_5$) will be represented using the following diagrams. The vertical line represents $\PGG(1,q)$ and on it the number of $\F$-rational points of the cubic are represented. Simple points are represented by bullets, double points by a bullet and a circle, and triple points by a bullet and two circles.
\begin{figure}[!h]
\centering
\begin{tikzpicture}[scale=0.4]
\draw (-6,0)--(-6,4);
\draw[fill=black] (-6,3) circle (3.0 pt);
\draw (-6,3) circle (6.0 pt);
\draw (-6,3) circle (9.0 pt);
\draw (-6,-1) node (a) [label=below:$\cH_1$]{};
\draw (-3,0)--(-3,4);
\draw[fill=black] (-3,3) circle (3.0 pt);
\draw (-3,3) circle (6.0 pt);
\draw[fill=black] (-3,2) circle (3.0 pt);
\draw (-3,-1) node (a) [label=below:$\cH_2$]{};
\draw (0,0)--(0,4);
\draw[fill=black] (0,3) circle (3.0 pt);
\draw[fill=black] (0,2) circle (3.0 pt);
\draw[fill=black] (0,1) circle (3.0 pt);
\draw (0,-1) node (a) [label=below:$\cH_3$]{};
\draw (0,0)--(0,4);
\draw[fill=black] (0,3) circle (3.0 pt);
\draw[fill=black] (0,2) circle (3.0 pt);
\draw[fill=black] (0,1) circle (3.0 pt);
\draw (0,-1) node (a) [label=below:$\cH_3$]{};
\draw (3,0)--(3,4);
\draw[fill=black] (3,3) circle (3.0 pt);
\draw[fill=black] (3,2) node[cross=2.0 pt]{};
\draw[fill=black] (3,1) node[cross=2.0 pt]{};
\draw (3,-1) node (1) [label=below:$\cH_4$]{};
\draw (6,0)--(6,4);
\draw[fill=black] (6,3) node[cross=2.0 pt]{};
\draw[fill=black] (6,2) node[cross=2.0 pt]{};
\draw[fill=black] (6,1) node[cross=2.0 pt]{};
\draw (6,-1) node (1) [label=below:$\cH_5$]{};
\end{tikzpicture}
\caption{The five $G$-orbits of planes in $\PG$}
\end{figure}
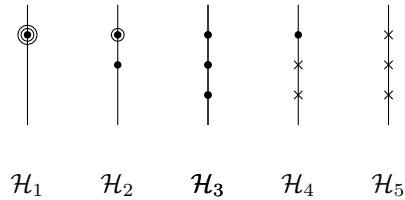
The position of the points on the vertical lines are meaningless and will be changed at will. For example, the $G$-orbit $\cH_2$ may also be represented by its diagram above with the positions of the simple point and the double point interchanged.

 The $5$ $G$-orbits of points in $\PGG(3,q)$ are as follows.
\begin{itemize}
\item[] $\mathcal{P}_1$. Points on $\C$ corresponding to osculating planes of $\C$ under $\sigma$.
\item[] $\mathcal{P}_2$. Points on tangent lines of the twisted cubic (not on the twisted cubic) correspond to planes containing exactly $2$ points of $\C$ under $\sigma$.
\item[] $\mathcal{P}_3$. Points on three osculating planes correspond to planes through $3$ distinct points of $\C$ under $\sigma$.
\item[] $\mathcal{P}_4$. Points not on $\C$ which lie on exactly one osculating plane correspond to planes through exactly one point of $\C$ under $\sigma$.
\item[] $\mathcal{P}_5$. Finally, the points which do not lie on any osculating plane correspond to planes containing no point of $\C$ under $\sigma$.
\end{itemize}

We are now in a position to define the combinatorial invariants in which we are interested.

\begin{definition}
The \emph{point\;orbit\;distribution} of a line $\ell$ in $\PG$, is denoted by
$OD_0(\ell)$ is a list $[a_0,b_0,c_0,d_0$, $e_0]$ where the entries $a_0$, $b_0$, $c_0$, $d_0$ and $e_0$ are the number of points on $\ell$ in the $G$-orbits on points of $\PGG(3,q)$ as listed above.
Similarly, we define the \emph{plane\;orbit\;distribution} of a line $\ell$ in $\PG$, which we denote by $OD_2(\ell)=[a_2,b_2,c_2,d_2,e_2]$. Clearly, we must have $a_{2}+b_{2}+c_{2}+d_{2}+e_{2}=a_{0}+b_{0}+c_{0}+d_{0}+e_{0}=q+1$.
\end{definition}

Each of the $G$-orbits of lines in $\PGG(3,q)$ is denoted by $\cL_i$ and the $G$-orbit of $\ell^\sigma$ with $\ell \in {\mathcal{L}}_i$ is denoted by $\mathcal{L}^\perp_i$.
For each $G$-orbit, we will include a diagram representing the corresponding pencils of cubics, say $\cP(\Cu_1,\Cu_2)$, on $\PGG(1,q)$.
The diagram consists of two vertical lines, each representing a copy of $\PGG(1,q)$ with its cubic $\Cu_i$, $i=1,2$, as explained above.
Horizontal dashed lines are used to indicate that the point on the two copies of $\PGG(1,q)$ coincide. For example the pencil $\cP(\Cu_1,\Cu_2)$ with $\Cu_1=\cZ(X_0^2X_1)$ and $\Cu_2=\cZ(X_0(X_0+X_1)^2)$ is represented by the following diagram.
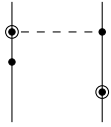
\begin{figure}[!h]
\centering
\begin{tikzpicture}[scale=0.4]
\draw (-3,0)--(-3,4);
\draw[fill=black] (-3,3) circle (3.0 pt);
\draw (-3,3) circle (6.0 pt);
\draw[fill=black] (-3,2) circle (3.0 pt);
\draw (0,0)--(0,4);
\draw[fill=black] (0,3) circle (3.0 pt);
\draw[fill=black] (0,1) circle (3.0 pt);
\draw (0,1) circle (6.0 pt);
\draw [dashed] (-2.7,3)--(-0.3,3);
\end{tikzpicture}
\end{figure}

The corresponding $G$-orbit of lines in $\PG$ will be represented by the same diagram.
Obviously different diagrams may represent the same $G$-orbit. As we will see, some diagrams represent a unique $G$-orbit of lines in $\PG$, while others represent several $G$-orbits.

The set of squares of $\F$ is denoted by $\square$ and the set of non-squares of $\F$ is denoted by $\triangle$.

\section{Lines contained in osculating planes}


We start our classification with the lines which are the intersection of two osculating planes of the twisted cubic. The fact that these lines form one $G$-orbit is well known. For the sake of completeness, we have included an argument for this fact in the proof of the following lemma. The main contribution of the lemma is the determination of the point and plane orbit distributions of such lines.

\begin{lemma} \label{extinosc} There is one orbit $\cL_1=\mathcal{O}^\perp_1$ of external lines contained in the intersection of osculating planes of the twisted cubic $\C$ in $\PG$. A line of $\ell \in \cL_1$ has plane orbit distribution $OD_2(\ell)=[2,0,0,(q-1),0]$ and point orbit distribution $OD_0(\ell)=[0,2,(q-1),0,0]$ for $q\equiv 5$ (mod $6$),
and $OD_2(\ell)=[2,0,\frac{(q-1)}{3},0,\frac{2(q-1)}{3}]$ and $OD_0(\ell)=[0,2,(q-1),0,0]$ for $q\equiv 1$ (mod $6$).
\end{lemma}
\begin{proof} Consider the pencil $\mathcal{P}(\Cu_1,\Cu_2)$ with \[\Cu_1=\mathcal{Z}(X_0^{3})\;\;\;\;\;\mbox{and}\;\;\;\Cu_2=\mathcal{Z}(X_1^{3}).\]	
The line $\ell$ is determined by the image of cubics $\Cu_1$, $\Cu_2$ under $\delta_3$. Thus, $\ell=\delta_3(\Cu_1)\cap\delta_3(\Cu_2)=\mathcal{Z}(Y_0,Y_3)$.	In order to determine the point orbit distribution of $\ell$, firstly we try to find the points on $\ell$ contained in the osculating planes of the twisted cubic $\C$. Consider
\[\ell \cap \Pi(t)=\mathcal{Z}(-t^3Y_0+3t^2Y_1-3tY_2+Y_3,Y_0,Y_3)\] where $\Pi(t)$ is the osculating plane of $\C$ at the point $P(t)$. Then the points of $\ell$ contained in $\Pi(t)$ are of the form $(0,1,t,0)$ for $t\in \F$.\\
(1) If $t=0$, then the point $(0,1,0,0)$ is on the tangent line of $\C$ at $P(0)$.\\
(2) If $t\neq 0$, then the point $(0,1,t,0)\in \Pi(t)$,  $\Pi(\infty)=\mathcal{Z}(Y_0)$, $\Pi(0)=\mathcal{Z}(Y_3)$. If the point $(0,1,s,0)\in \ell$ is also contained in  $\Pi(t)$, then \[3t^2-3ts=(3t)(t-s)=0.\] Since $t$ is nonzero, $t=s$. Therefore, each point of the form $(0,1,t,0)$ is contained in three osculating planes.\\
(3) The remaining point $(0,0,1,0)\in \ell$ is on the tangent line of $\C$ at $P(\infty)$.\\
\par In total, the line $\ell$ contains $2$ points from the point orbit $\mathcal{P}_2$ and $(q-1)$ points from the point orbit $\mathcal{P}_3$. It follows that the point orbit distribution of $\ell$ is $OD_0(\ell)=[0,2,(q-1),0,0]$.
\bigskip
\par To determine the plane orbit distribution of $\ell$, consider a plane $\Pi= \langle\ell, P(t)\rangle$. Then $\Pi$ has equation $\mathcal{Z}(t^3Y_0-Y_3)$. If the point $(1,s,s^2,s^3)$ of $\C$ is contained in $\Pi$, then $t^3=s^3$. Therefore, defining the map \[\begin{array}{cccc}
                                                                                                            \psi: & \mathbb{F}^*_q & \rightarrow &  \mathbb{F}^*_q  \\
                                                                                                             & t & \mapsto & t^3
                                                                                                          \end{array}\]
the size of the image $|\Ima(\psi)|$ gives the number of planes of the form $\mathcal{Z}(t^3Y_0-Y_3)$.

If $3\nmid (q-1)$, then $\psi$ is bijective. If $q\equiv 5$ (mod $6$) then $\ell$ lies on $(q-1)$ distinct planes containing $1$ point of the twisted cubic $\C$. The planes $\Pi(\infty)=\mathcal{Z}(Y_0)$ and $\Pi(0)=\mathcal{Z}(Y_3)$ are the osculating planes of the twisted cubic $\C$ containing $\ell$.
The plane orbit distribution of $\ell$ is therefore $OD_2(\ell)=[2,0,0,(q-1),0]$ for $q\equiv 5$ (mod $6$).

\par If $q\equiv 1$ (mod $6$), then $3\; |\; (q-1)$. In this case $\ell$ lies on $(q-1)/3$ planes containing $3$ points of the twisted cubic $\C$ and there are $2(q-1)/3$ planes containing no point of the twisted cubic $\C$. In total there are $2$ planes from the plane orbit $\mathcal{H}_2$, $(q-1)/3$ planes from the plane orbit $\mathcal{H}_3$ and $2(q-1)/3$ planes from the plane orbit $\mathcal{H}_5$ containing the line $\ell$. Hence, the plane orbit distribution of $\ell$ is $OD_2(\ell)=[2,0,\frac{(q-1)}{3},0,\frac{2(q-1)}{3} ]$ for $q\equiv 1$ (mod $6$).

\par The fact that $G$ acts transitively on this set of external lines follows from the fact that the corresponding pencils are completely determined by two points on $\PGG(1,q)$.
\end{proof}

%

\begin{figure}[!h]
\centering
\begin{tikzpicture}[scale=0.4]
\draw (-3,0)--(-3,4);
\draw[fill=black] (-3,3) circle (3.0 pt);
\draw (-3,3) circle (6.0 pt);
\draw (-3,3) circle (9.0 pt);
\draw (-3,-1) node {};
\draw (0,0)--(0,4);
\draw[fill=black] (0,1) circle (3.0 pt);
\draw (0,1) circle (6.0 pt);
\draw (0,1) circle (9.0 pt);
\draw (0,1) node {};
\end{tikzpicture}
\vspace{-1em}
\caption{A pencil of cubics corresponding to $\cL_1$}
\end{figure}
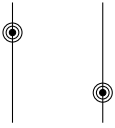


In the following lemma we consider the tangent lines of the twisted cubic.
\begin{lemma} There is one orbit $\cL_2=\mathcal{O}_2$ of tangent lines to the twisted cubic $\C$ in $\PG$. A tangent line of the twisted cubic $\C$ has plane orbit distribution $OD_2(\ell)=[1,q,0,0,0]$ and point orbit distribution $OD_0(\ell)=[1,q,0,0,0]$.
\end{lemma}
\begin{proof} Since $G$ acts transitively on the points of $\C$ it follows that $G$ also acts transitively on the tangent lines.
Consider the pencil $\mathcal{P}(\Cu_1,\Cu_2)$ with
$\Cu_1=\mathcal{Z}(X_1^{3})$ and $\Cu_2=\mathcal{Z}(X_1^{2}X_0)$.	
The line $\ell$ is given by $\ell=\delta_3(\Cu_1)\cap\delta_3(\Cu_2)=\mathcal{Z}(Y_2,Y_3)$. The line $\ell$ is the tangent line of $\C$ at the point $(1,0,0,0)$. It implies that the line contains $1$ point from the point orbit $\mathcal{P}_1$ and $q$ points from the point orbit $\mathcal{P}_2$. Then $\ell$ has a point orbit distribution $OD_0(\ell)=[1,q,0,0,0]$. By Lemma~\ref{selfdual}, the line $\ell$ of $\C$ is self polar. Therefore, $OD_0(\ell)=OD_2(\ell)$.\\

\par The union of the cubics of the pencil $\mathcal{P}(\Cu_1,\Cu_2)$ consists of two points of $\PGG(1,q)$. Since $G$ acts transitively on two points of $\mathcal{P}(\Cu_1,\Cu_2)$, each two such pencils are projectively equivalent. Therefore, there is one orbit of lines in $\PG$ with the given plane and point orbit distributions.
\end{proof}

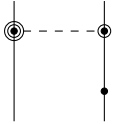
\begin{figure}[!h]
\centering
\begin{tikzpicture}[scale=0.4]
\draw (-3,0)--(-3,4);
\draw[fill=black] (-3,3) circle (3.0 pt);
\draw (-3,3) circle (6.0 pt);
\draw (-3,3) circle (9.0 pt);
\draw (-3,-1) node {};
\draw (0,0)--(0,4);
\draw[fill=black] (0,3) circle (3.0 pt);
\draw (0,3) circle (6.0 pt);
\draw[fill=black] (0,1) circle (3.0 pt);
\draw [dashed] (-2.7,3)--(-0.3,3);
\end{tikzpicture}
\vspace{-1em}
\caption{A pencil of cubics corresponding to $\cL_2$}
\end{figure}


The following lemma describes the orbits consisting of non-tangent unisecants contained in osculating planes of the twisted cubic. They form one orbit. Since these lines are self polar (see Lemma~\ref{selfdual}) their point orbit distribution is equal to their plane orbit distribution.

\begin{lemma} There is one orbit $\cL_3=\mathcal{O}_4$ of non-tangent unisecants in osculating planes of the twisted cubic $\C$ in $\PG$ with the same plane and point orbit distributions $OD_2(\ell)=OD_0(\ell)=[1,1,\frac{(q-1)}{2},\frac{(q-1)}{2},0]$.
\end{lemma}
\begin{proof} Consider the non-tangent unisecant line $\ell=\mathcal{Z}(Y_1,Y_3)$.
Then $\ell$ corresponds to the pencil $\mathcal{P}(\Cu_1,\Cu_2)$ with
\[\Cu_1=\mathcal{Z}(X_1^{3})\;\;\;\;\;\mbox{and}\;\;\;\Cu_2=\mathcal{Z}(X_0^{2}X_1).\]

We first determine which points on $\ell$ are contained in osculating planes of the twisted cubic $\C$.
The intersection of $\ell$ with an osculating plane is given by
\[\ell \cap \Pi(t)=\mathcal{Z}(-t^3Y_0+3t^2Y_1-3tY_2+Y_3,Y_1,Y_3)\] where $\Pi(t)$ is an osculating plane of $\C$ at the point $P(t)$. The points of $\ell$ which are contained in $\Pi(t)$ are of the form $(1,0,\frac{-t^2}{3},0)$ for $t\in \F$.
\begin{enumerate}
\item If $t=0$, then the point $(1,0,0,0)$ is a point of $\C$.
\item Assume that $t\neq 0$. If the point $(1,0,s,0)\in \ell$ is contained in $\Pi(t)$, then \[-t^3-3ts=-t(t^2+3s)=0.\] This implies $t^2=-3s$. If $-3s\in \Box$ for $s\in \F\setminus\{0\}$, then the point $(1,0,s,0)\in \ell$ is contained in three osculating planes $\Pi(t_1)$, $\Pi(t_2)$ and $\Pi(0)$, for $t_1\neq 0$, $t_2\neq 0$, $t_1^2=t_2^2=-3s$.
If $-3s\in \triangle$, then the point $(1,0,s,0)\in \ell$ is contained in exactly one osculating plane $\Pi(0)$. This gives $(q-1)/2$ points of $\ell$ contained in three osculating planes and $(q-1)/2$ points of $\ell\setminus \C$ contained in just one osculating plane.
\item The point $(0,0,1,0)$ is a point on the tangent line $\mathcal{Z}(Y_0,Y_1)$ at $P(\infty)$.
\end{enumerate}
\par Summarizing these cases, we may conclude that $OD_0(\ell)=[1,1,\frac{(q-1)}{2},\frac{(q-1)}{2},0]$. By Lemma~\ref{selfdual}, the line $\ell$ of $\C$ is self polar. Therefore, $OD_0(\ell)=OD_2(\ell)$.
\bigskip
\par Since the pencil $\mathcal{P}(\Cu_1,\Cu_2)$ is completely determined by just two points of $\PGG(1,q)$, each two such pencils are projectively equivalent.  Hence, there is one orbit of lines in $\PG$ with the given plane and point orbit distributions.
\end{proof}

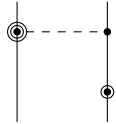
\begin{figure}[!h]
\centering
\begin{tikzpicture}[scale=0.4]
\draw (-3,0)--(-3,4);
\draw[fill=black] (-3,3) circle (3.0 pt);
\draw (-3,3) circle (6.0 pt);
\draw (-3,3) circle (9.0 pt);
\draw (-3,-1) node {};
\draw (0,0)--(0,4);
\draw[fill=black] (0,3) circle (3.0 pt);
\draw[fill=black] (0,1) circle (3.0 pt);
\draw (0,1) circle (6.0 pt);
\draw [dashed] (-2.7,3)--(-0.3,3);
\end{tikzpicture}
\vspace{-1em}
\caption{A pencil of cubics corresponding to $\cL_3$}
\end{figure}


The next type of lines which we consider are the lines belonging to $\cO_5^\perp$, which are lines in osculating planes which do not intersect the twisted cubic and which are not contained in two distinct osculating planes. We will show that $\cO_5^\perp$ splits into different $G$-orbits. The point orbit distributions and plane orbit distributions depend on $-3$ being a square or not in $\F$. Moreover, both orbit distributions serve as complete combinatorial invariants for these orbits.

\begin{lemma} If $-3$ is a nonsquare in $\F$ then there is one orbit $\cL_4\subset\mathcal{O}^\perp_5$ of external lines in osculating planes of the twisted cubic $\C$ in $\PG$ with orbit distributions
$$
OD_2(\ell)=[1,2,\frac{(q-5)}{6},\frac{(q-3)}{2},\frac{(q+1)}{3}] \mbox{ and } OD_0(\ell)=[0,3,\frac{(q-3)}{2},\frac{(q-1)}{2},0].
$$
If $-3$ is a square in $\F$ then there is one orbit $\cL_4\subset\mathcal{O}^\perp_5$ with orbit distributions
$$
OD_2(\ell)=[1,2,\frac{(q-7)}{6},\frac{(q-1)}{2},\frac{(q-1)}{3}] \mbox{ and } OD_0(\ell)=[0,3,\frac{(q-3)}{2},\frac{(q-1)}{2},0].
$$
The number of lines in the $G$-orbit $\cL_4$ is $q(q^2-1)$.
\end{lemma}

\begin{proof} Consider the pencil $\mathcal{P}(\Cu_1,\Cu_2)$ with
$\Cu_1=\mathcal{Z}(X_0^{3})$, and $\Cu_2=\mathcal{Z}(X_1^{2}(X_0-X_1))$,
with diagram shown in Fig.\ref{fig:cl4}.
%
%
%
%
The corresponding line $\ell$ in $\PG$ is given by $\delta_3(\Cu_1)\cap\delta_3(\Cu_2)=\mathcal{Z}(Y_0,Y_2-Y_3)$.
Firstly, we look at osculating planes through points on $\ell$. We have
\[\ell \cap \Pi(t)=\mathcal{Z}(-t^3Y_0+3t^2Y_1-3tY_2+Y_3,Y_0,Y_2-Y_3)\] where as before $\Pi(t)$ is the osculating plane of $\C$ at the point $P(t)$.
The points of $\ell$ which are contained in $\Pi(t)$ are of the form $(0,3t-1,3t^2,3t^2)$ for $t\in \F$.

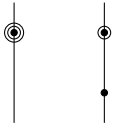
\begin{figure}[!h]
\centering
\begin{tikzpicture}[scale=0.4]
\draw (-3,0)--(-3,4);
\draw[fill=black] (-3,3) circle (3.0 pt);
\draw (-3,3) circle (6.0 pt);
\draw (-3,3) circle (9.0 pt);
\draw (-3,-1) node {};
\draw (0,0)--(0,4);
\draw[fill=black] (0,3) circle (3.0 pt);
\draw (0,3) circle (6.0 pt);
\draw[fill=black] (0,1) circle (3.0 pt);
\end{tikzpicture}
\vspace{-1em}
\caption{A pencil of cubics corresponding to $\cL_4$}\label{fig:cl4}
\end{figure}

If $t=0$, then we get the point $(0,1,0,0)$ which is on the tangent line $\mathcal{Z}(Y_2,Y_3)$ of $\C$ at $P(0)$, and for $t=\frac{1}{3}$, we get the point $(0,0,1,1)$ on the tangent line $\mathcal{Z}(Y_0,Y_1)$ of $\C$ at $P(\infty)$.

If the point $(0,1,s,s)$ of $\ell$ is contained in the osculating plane $\Pi(t)$ for $s\neq0$, then $s$ and $t$ must satisfy
\[3t^2-3st+s=0.\]
This is a quadratic equation in $t$, with discriminant
\[D_s=9s^2-12s=3s(3s-4).\]

For $s=\frac{12}{9}$, $D_s=0$, and the point $(0,1,\frac{12}{9},\frac{12}{9})$ is contained in the osculating planes $\Pi(\frac{2}{3})$, and $\Pi(\infty)$ where $t=\frac{s}{2}$. Therefore, the point $(0,1,\frac{12}{9},\frac{12}{9})$ lies on the tangent line of $\C$ at $P(\frac{2}{3})$.

If $D_s=9s^2-12s$ is a nonzero square, then there are three osculating planes $\Pi(t_1)$, $\Pi(t_2)$, and $\Pi(\infty)$ containing the point $(0,1,s,s)$ of $\ell$ where $t_{1,2}=\frac{3s\pm \sqrt{9s^2-12s}}{6}$. To determine how many values of $s$, $D_s\in \Box$, consider the equation \[9s^2-12s=y^2.\]
Homogenizing gives the quadratic form
\[g(s,y,z)=9s^2-12sz-y^2.\]
The number of points of the form $(s,y,0)$ on the conic $\cZ(g)$ is $2$ and the number of points of the form $(s,0,1)$ is also $2$. Therefore, there are $(q-3)/2$ values of $s$, for which the equation $9s^2-12s$ is a nonzero square in $\F$.

If $D_s\in \triangle$, then the point $(0,1,s,s)$ is contained in exactly one osculating plane $\Pi(\infty)$. By the above, there remain \[q-\frac{q+1}{2}=\frac{q-1}{2}\] values of $s$ for which $D_s\in \triangle$.

In total, the line $\ell$ contains $3$ points from the point orbit $\mathcal{P}_2$, $(q-3)/2$ points from the point orbit $\mathcal{P}_3$, and $(q-1)/2$ points from the point orbit $\mathcal{P}_4$.
Hence, the point orbit distribution is  $OD_0(\ell)=[0,3,\frac{(q-3)}{2},\frac{(q-1)}{2},0]$.

\bigskip

Now, in order to determine the plane orbit distribution of $\ell$, consider the plane
$$\Pi=\langle\ell, P(t)\rangle,$$
where $P(t)$ is a point of $\C$.
Then $\Pi=\langle\ell, P(\infty)\rangle$ is the osculating plane $\mathcal{Z}(Y_0)$, and
$\Pi=\mathcal{Z}((t^3-t^2)Y_0+Y_2-Y_3)$ for $t\in \F$.
 \bigskip

 If $t^2(t-1)=0$, then $\mathcal{Z}(Y_2-Y_3)$ is a plane containing the points $(1,0,0,0)$ and $(1,1,1,1)$ of the twisted cubic $\C$. This plane belongs to $\mathcal{H}_2$.
  \item Assume that $t^2(t-1)\neq0$. If $\Pi$ contains another point $(1,s,s^2,s^3)$ of $\C$, then \[t^3-t^2=s^3-s^2.\] Since $s\neq t$, this gives $f(s)=s^2+s(t-1)+t^2-t=0$, which is a quadratic equation in $s$. Note that $s=t$ is a solution of $f(s)=0$ if and only if $t=0$ or $t=\frac{2}{3}$. The discriminant of $f(s)$ is
 \[D_t=(t-1)^2-4(t^2-t)=-3t^2+2t+1=(1-t)(3t+1).\]

\bigskip

 If $t^3-t^2\neq 0$ and $D_t=0$, then $\Pi$ is the plane containing $2$ points $(1,t,t^2,t^3)$, and $(1,s_1,s_1^2,s_1^3)$ of the twisted cubic $\C$. Since $t^3-t^2\neq 0$, $D_t=0$ for $t=-\frac{1}{3}$. Then $\Pi$ is the plane containing $(1,-\frac{1}{3},(-\frac{1}{3})^2, (-\frac{1}{3})^3)$ and $(1,\frac{2}{3},(\frac{2}{3})^2, (\frac{2}{3})^3)$, and therefore $\Pi\in \mathcal{H}_2$.

\bigskip

 If $t^3-t^2\neq 0$ and $D_t=-3t^2+2t+1$ is a nonzero square in $\F\setminus \{0,1\}$ then
the two solutions of $f(s)=0$ are
 \[s_{1,2}=\frac{(1-t)\pm\sqrt{-3t^2+2t+1}}{2}.\]
In this case the plane $\Pi$ contains $3$ points $(1,s_1,s_1^2,s_1^3)$, $(1,s_2,s_2^2,s_2^3)$, and $(1,t,t^2,t^3)$ of the twisted cubic $\C$, i.e. $\Pi\in \mathcal{H}_3$.

\bigskip

 Next, we determine for how many of the remaining values of $t \in \F$, $D_t$ is a nonzero square.
By the condition on $D_t$, it is convenient to consider the conic $\mathcal{Z}(g)$ in $\PGG(2,q)$ defined by \[g(X,Y,Z)=-3X^2+2XZ+Z^2-Y^2.\] Then $\mathcal{Z}(g)$ is a non-degenerate conic since it has no singular point. The points on the conic $\mathcal{Z}(g)$ can be parameterized as \[\{(x,y,1), (x,-y,1)\;|\;-3x^2+2x+1\in \Box\}\cup\{(1,-\sqrt{-3},0),(1,\sqrt{-3},0)\}\] where $y=\sqrt{-3x^2+2x+1}$.

\begin{enumerate}
\item A point of the form $(0,y,1)$ lies on the conic $\mathcal{Z}(g)$ if and only if $y^2=1$. Hence, there are $2$ points $(0,1,1)$ and $(0,-1,1)$ of the given form. They both correspond to $t=0$.
\item A point of the form $(x,0,1)$ lies on the conic $\mathcal{Z}(g)$ if and only if $-3x^2+2x+1=(1-x)(3x+1)=0$. In this case we find the points $(1,0,1)$ and $(-\frac{1}{3},0,1)$ of the given form. Also these points do not contribute to the values of $t$ for which $D_t$ is a nonzero square in $\F$.
\item A point of the form $(1,y,0)$ lies on the conic $\mathcal{Z}(g)$ if and only if $-3=y^2$. There exists a point of the given form if and only if $-3$ is a square in $\F$. Therefore, there are either $2$ points $(1,\sqrt{-3},0)$ and $(1,-\sqrt{-3},0)$ or no points of the form $(1,y,0)$. Again these points do not contribute to the values of $t$ for which $D_t$ is a nonzero square in $\F$.
\item A point of the form $(x,y,1)$ with $xy\neq0$ lies on the conic $\mathcal{Z}(g)$ if and only if $-3x^2+2x+1=y^2$. From the previous parts we have  already either $4$ or $6$ points on the conic $\mathcal{Z}(g)$. For $x=\frac{2}{3}$, which corresponds to the value $t=\frac{2}{3}$ in part of the proof determining the point orbit distribution, we must have $\sqrt{-3x^2+2x+1}=1$, and we get $2$ points $(\frac{2}{3},1,1)$ and $(\frac{2}{3},-1,1)$. There remain either $(q-5)/2$ or $(q-7)/2$ pairs on the conic $\mathcal{Z}(g)$ of the form $(x,y,1)$ and $(x,-y,1)$.
\end{enumerate}
Consequently, there are either $(q-5)/2$ or $(q-7)/2$ values of $t$ for which $D_t$ is a nonzero square in $\F$. For these values of $t$, there is a plane $\Pi$ through $\ell$ containing $3$ points $(1,s_1,s_1^2,s_1^3)$, $(1,s_2,s_2^2,s_2^3)$, $(1,t,t^2,t^3)$ of the twisted cubic $\C$. Since each such plane is counted three times, there are either $(q-5)/6$ or $(q-7)/6$ such planes.

\bigskip

 If $t^3-t^2\neq 0$ and $D_t=-3t^2+2t+1$ is a non-square in $\F$, then $\Pi$ is the plane containing $1$ point $(1,t,t^2,t^3)$ of the twisted cubic $\C$. We know that for $t=1$ and $t=-\frac{1}{3}$, $D_t=0$ and for $t=0$ and $t=\frac{2}{3}$, $D_t=1$. We have determined that there are either $(q-5)/2$ or $(q-7)/2$ nonzero square values of $D_t\neq 1$, hence the number of non-square values of $D_t$ is either \[q-\frac{(q+3)}{2}=\frac{(q-3)}{2}\;\;\mbox{or}\;\;q-\frac{(q+1)}{2}=\frac{(q-1)}{2}.\] Therefore, if $-3$ is a square in $\F$, then there are $(q-1)/2$ planes through $\ell$ containing $1$ point of the twisted cubic $\C$, otherwise, there are $(q-3)/2$ such planes.

\bigskip

Summarizing all the results, for the plane orbit distribution we found $[a_2,b_2,c_2,d_2,e_2]$ with
$a_2=1$, $b_2=2$, $c_2$ is either $(q-5)/6$ or $(q-7)/6$ and $d_2$ is either $(q-3)/2$ or $(q-1)/2$.
Since $e_2=q+1-(a_2+b_2+c_2+d_2)$, $e_2$ is $(q+1)/3$, if $-3\in \triangle$, and $e_2$ is $(q-1)/3$, if $-3 \in \Box$. This determines the plane orbit distribution of the line $\ell$.
\bigskip

Clearly, each pencil of cubics in $\PGG(1,q)$ which contains a cubic of type 1 and a cubic of type 2 can be mapped to any other such pencil by the unique projectivity of $\PGG(1,q)$ mapping the two frames of $\PGG(1,q)$ determined by the two pairs of cubics (one of type 1 and one of type 2) onto each other. It follows that, up to equivalence, there is a unique such pencil. Equivalently, there is a unique $G$-orbit of lines in $\PG$ with the given orbit distributions.

\bigskip

Counting triples $(\ell,\Pi_1,\Pi_2)$ with $\ell\in \cL_4$, $\Pi_1\in \cH_1$, and $\Pi_2\in \cH_2$, we get
$$
|\cL_4|~2=(q+1)q(q-1)
$$
where the left hand side follows from the fact that a line $\ell\in \cL_4$ is contained in a unique plane from $\cH_1$ and in two planes from $\cH_2$.
The right hand side follows from the fact that there are $q+1$ choices for a plane $\Pi_1\in \cH_1$ and $q(q-1)$ choices for a plane $\Pi_2\in \cH_2$ which does not pass through the contact point of $\Pi_1$. The number of lines in this $G$-orbit is therefore equal to $q(q^2-1)$.
\end{proof}

\begin{remark}\label{rem:size_L_4}
Note that the total number of lines in the $G$-orbit is equal to ${q(q^2-1)}/{2}$, since such a line is determined by the choice of a cubic of type 1 and a cubic of type 2 disjoint from it, and for each such pencil there are two choices for the cubic of type 2. It follows that the number of lines in an osculating plane belonging to the $G$-orbit $\cL_4$ is
$$
\frac{q(q^2-1)/2}{q+1}=\frac{q(q-1)}{2},
$$
which follows from the total number of lines in $\cL_4$ and the fact that each line of $\cL_4$ is contained in exactly one osculating plane, plus the fact that each osculating plane must contain the same number of lines from $\cL_4$.
\end{remark}

Recall that the cross ratio of a $4$-tuple of points $(P_1,P_2,P_3,P_4)$ with affine coordinates $x_1$, $x_2$, $x_3$ and $x_4$ with respect to a frame $\Lambda$ of $\PGG(1,q)$ is given by
\[(P_1,P_2;P_3,P_4)=\frac{(x_3-x_1)(x_4-x_2)}{(x_3-x_2)(x_4-x_1)}.\]
For any quadruple $T=\{P_1,P_2,P_3,P_4\}$ of distinct points of $\PGG(1,q)$,
define its $J$-invariant as
\[J(T)=\sum\limits_{i<j}u_iu_j.\]
where the $u_i$'s are the (in general 6) cross ratios which can be obtained from $T$. Abusing notation, we also define
\[J(X)=6-\frac{(X^2-X+1)^3}{X^2(X-1)}.\]
It is well known that $J(T)=J(u_i)$ for each $u\in \{u_1, u_2, u_3, u_4, u_5, u_6\}$.
For a pair of cubics $(\Cu_1,\Cu_2)$ on $\PGG(1,q)$, where $\Cu_1$ cubic of type $1$ and $\Cu_2$ cubic of type $3$, define
$J(\Cu_1,\Cu_2)$ as $J(T)$ where $T=\{P_1,P_2,P_3,P_4\}$ with $P_4$ the triple point of $\Cu_1$, and $P_1$, $P_2$, $P_3$ are the points of $\Cu_2$.

\begin{lemma}\label{notype$2$}
Let $\cP(\Cu_1,\Cu_2)$ be a pencil with empty base, with $\Cu_1$ of type $1$ and $\Cu_2$ of type $3$.
There exists $u \in \F$ such that $J(u)=J(\Cu_1,\Cu_2)$ satisfying $u^2-u+1$ is a non-square in $\F$ if and only if
$\cP(\Cu_1,\Cu_2)$ contains no cubic of type $2$.
\end{lemma}
\begin{proof}
Consider a pencil $\mathcal{P}(\Cu_1,\Cu_2)$ satisfying the hypotheses. Then without loss of generality we may assume that
\[\Cu_1=\mathcal{Z}((X_1-uX_0)^3)\;\;\mbox{and}\;\;\Cu_2=\mathcal{Z}(X_0X_1(X_0-X_1)),\]
where $u^2-u+1$ is a non-square in $\F$.
A cubic $\Cu_s \in\mathcal{P}(\Cu_1,\Cu_2)$ is of the form $\Cu_s=\mathcal{Z}(f_s)$ for $s\in \F\cup\{\infty\}$, where
\[f_s(X_0,X_1)=X_1^3-3uX_1^2X_0+3u^2X_1X_0^2-u^3X_0^3+s(X_0^2X_1-X_0X_1^2).\]
By way of contradiction, suppose $\Cu_s$ is of type 2. Then $s\neq 0$ and
$f_s(1,X_1)$ must have repeated roots in $\F$.
The substitution $X_1\mapsto X_1-(-3u-s)/3$ gives
\[f_s(1,X_1)=X_1^3+\left (-2su+s-\frac{s^2}{3}\right )X_1+\left (\frac{-2s^3-18s^2u+9s^2-27su^2+27su}{27}\right )\]
where 
\[k=-2su+s-\frac{s^2}{3}\;\;\mbox{and}\;\;m=\frac{-2s^3-18s^2u+9s^2-27su^2+27su}{27}.\]
Since $f_s(1,X_1)$ has repeated roots, the discriminant $D_{f_s}=0$, where
\[D_{f_s}=-4k^3-27m^2=s^2(s^2+(-2(2u-1)(u-2)(u+1))s-27u^2(u-1)^2).\]
But since the discriminant
 of the polynomial
\[g_u(s)=s^2+(-2(2u-1)(u-2)(u+1))s-27u^2(u-1)^2\]
is equal to
 \[D_{g_u}=16(u^2-u+1)^3\]
this contradicts that fact that
$u^2-u+1$ is a non-square in $\F$.
\end{proof}

\begin{lemma}
The line $\ell=\delta_3(\Cu_1)\cap\delta_3(\Cu_2)$ with $\Cu_1=\mathcal{Z}((X_1-uX_0)^{3})$ and $\Cu_2=\mathcal{Z}(X_0X_1(X_0-X_1))$, where  $u^2-u+1$ is a non-square in $\F$, has
plane orbit distribution $OD_2(\ell)=[1,0,\frac{(q-1)}{6},\frac{(q+1)}{2},\frac{(q-1)}{3}]$ and point orbit distribution $OD_0(\ell)=[0,1,\frac{(q-1)}{2},\frac{(q+1)}{2},0]$ for $-3 \in \Box$, and $OD_2(\ell)=[1,0,\frac{(q+1)}{6},\frac{(q-1)}{2},\frac{(q+1)}{3}]$, $OD_0(\ell)=[0,1,\frac{(q-1)}{2},\frac{(q+1)}{2},0]$ for $-3\in \triangle$.
\end{lemma}

\begin{proof}
 The line $\ell=\delta_3(\Cu_1)\cap\delta_3(\Cu_2)$ is given by
 $$\ell=\mathcal{Z}(Y_3-3uY_2+3u^2Y_1-u^3Y_0,Y_1-Y_2),$$
and its points can be parameterized as
\[\{(1,s,s,(3u-3u^2)s+u^3)\;\;|\;\;s\in\F\}\cup\{(0,1,1,(3u-3u^2))\}.\]

($OD_0$) Firstly, we determine the points on $\ell$ which are contained in the osculating planes of the twisted cubic $\C$. We have
\[\ell \cap \Pi(t): -t^3+3t^2s-3ts+(3u-3u^2)s+u^3=0\] where $\Pi(t)$ is an osculating plane of $\C$ at the point $P(t)$. In order to solve this cubic equation we make the substitution $t\mapsto t+s$. This gives
\[f(t)=-t^3+(3s^2-3s)t+2s^3-3s^2-3su^2+3su+u^3\]
where
\[k=3s^2-3s\;\;\mbox{and}\;\;m=2s^3-3s^2-3su^2+3su+u^3.\]
The discriminant of $f(t)$ is
\[D_s=27(2su-s-u^2)^2(3s^2+2su-4s-u^2).\]
The roots of the polynomial $f(t)$ are
\[t_1=u-s,\;\;t_2=\frac{-\sqrt{9 s^2 + 6 s (u - 2) - 3 u^2} + s - u}{2},\;\;t_3=\frac{\sqrt{9 s^2 + 6 s (u - 2) - 3 u^2} + s - u}{2}.\]

(1) The discriminant $D_s$ is zero if and only if \[2su-s-u^2=0\;\;\mbox{or}\;\;3s^2+2su-4s-u^2=0.\] The equation $2su-s-u^2=0$ implies that $s(2u-1)=u^2$. This gives \[s=\frac{u^2}{2u-1}\;\;\mbox{for}\;\;u\neq \frac{1}{2}.\]
    \par The point $(1,s,s,(3u-3u^2)s+u^3)$ is contained in the osculating planes $\Pi(t_1)$, $\Pi(t_2)$ where \[t_{1}=u\;\;\mbox{and}\;\;t_2=-\frac{(2-u)u}{2u-1},\;\;\mbox{for}\;\;s=\frac{u^2}{2u-1}.\]

\par Note that for $u=\frac{1}{2}$, $u^2-u+1=\frac{3}{4}$. For $3\in \triangle$, $\frac{3}{4}\in \triangle$. Therefore, if $3\in \triangle$ and $u=\frac{1}{2}$, then $D_s=0$ for no values of $s$.
\par The equation \[3s^2+s(2u-4)-u^2=0\] has a solution if and only if its discriminant \[D_u=16(u^2-u+1)\] is either $0$ or a nonzero square in $\F$. By the hypothesis, $D_u$ is a non-square in $\F$.

\bigskip

(2) The discriminant $D_s$ is a nonzero square if and only if \[-3(-3s^2-2su+4s+u^2)\in \Box.\] This means that $-3$ and $(-3s^2-2su+4s+u^2)$ have the same quadratic residue in $\F$. In this case there are three osculating planes $\Pi(t_1)$, $\Pi(t_2)$, $\Pi(t_3)$ containing the point $(1,s,s,(3u-3u^2)s+u^3)$ of $\ell$.

\bigskip

\par Suppose that $-3\in \Box$. We want to determine for how many values of $s\in \F$, the discriminant $D_s$ is a nonzero square. Define \[g(S,Y,Z)=-3S^2-(2u-4)SZ+u^2Z^2-Y^2.\] Then $\mathcal{Z}(g)$ is a non-degenerate conic if and only if $u^2-u+1 \neq 0$, which is satisfied by the hypothesis. We have
\[\mathcal{Z}(g)=\{(s,y,1), (s,-y,1)\;\;|\;\;-3s^2-(2u-4)s+u^2\in \Box\}\cup\{(1,\sqrt{-3},0), (1,-\sqrt{-3},0)\}.\]

A point of the form $(0,y,1)$ lies on the conic $\mathcal{Z}(g)$ if and only if $u^2=y^2$. There are $2$ points $(0,u,1)$ and $(0,-u,1)$ on the conic.
A point of the form $(1,y,0)$ lies on the conic $\mathcal{Z}(g)$ if and only if $-3s^2=y^2$. Since $-3$ is a square in $\F$, there are $2$ such points $(1,\sqrt{-3},0)$ and $(1,-\sqrt{-3},0)$.
A point of the form $(s,y,1)$ with $sy\neq0$ lies on the conic $\mathcal{Z}(g)$ if and only if $-3s^2-(2u-4)s+u^2=y^2$. From the previous parts we have  already $4$ points on the conic $\mathcal{Z}(g)$. There remain $(q-3)/2$ pairs on the conic $\mathcal{Z}(g)$ of the form $(s,y,1)$ and $(s,-y,1)$.

\bigskip

In total, for $-3\in \Box$, the expression $-3s^2-(2u-4)s+u^2$ is a nonzero square in $\F$ for $(q-1)/2$ values of $s\in \F$. Therefore, if $-3\in \Box$, then $D_s\in \Box$ for $(q-1)/2$ values of $s$. Note that for $s=\frac{u^2}{2u-1}$ (with $u\neq \frac{1}{2}$) the point $(\frac{u^2}{2u-1},y,1)$ is on the conic $\mathcal{Z}(g)$ if and only if
\[g(\frac{u^2}{2u-1},y,1)=-3u^2(u-1)^2-y^2(2u-1)^2=0.\]
Since $u\neq 0,1$ it follows that if the point  $(\frac{u^2}{2u-1},y,1)$ lies on the conic $\mathcal{Z}(g)$ then $-3\in \Box$. Hence, $s=\frac{u^2}{2u-1}$ is one of the values of $s$ for which $D_s\in \Box$. We can conclude that there are three osculating planes containing $(1,s,s,(3u-3u^2)s+u^3)$ for $(q-3)/2$ values of $s$.

\bigskip

(3) If the discriminant $D_s$ is a non-square, then there exists an osculating plane $\Pi(u-s)$ containing the point $(1,s,s,(3u-3u^2)s+u^3)$ of $\ell$. The discriminant $D_s$ is a non-square if and only if $-3(-3s^2-2su+4s+u^2)\in \triangle$. From the previous discussions, we can easily say that if $-3\in \Box$, then for $(q+1)/2$ values of $s$, the expression $-3s^2-2su+4s+u^2$ is a non-square in $\F$.

\bigskip

(4) Finally consider the point $(0,1,1,(3u-3u^2))\in \Pi(\infty)$. It is contained in the osculating plane $\Pi(t)$ of $\C$ if and only if
\[3t^2-3t+(3u-3u^2)=-3(-t^2+t-u+u^2)=0.\]
It follows that the point $(0,1,1,(3u-3u^2))$ of $\ell$ is contained in the osculating planes
$\Pi(\infty)$, $\Pi(u)$, and $\Pi(1-u)$. The number of osculating planes through $(0,1,1,(3u-3u^2))$ is two for $1-u=u=\frac{1}{2}$, and three otherwise.

\bigskip

Collecting the results for $-3\in \Box$ and $u \neq \frac{1}{2}$, it follows that the line $\ell$ has $1$ point from the point orbit $\mathcal{P}_2$, $(q-1)/2$ points from the point orbit $\mathcal{P}_3$ and $(q+1)/2$ points from the point orbit $\mathcal{P}_4$.

\bigskip

Similarly, if $-3 \in \triangle$, then $-3s^2-(2u-4)s+u^2$ is nonzero square in $\F$ for $(q+1)/2$ values of $s\in \F$. There remain $q-(q+1)/2=(q-1)/2$ values of $s$ for which $-3s^2-(2u-4)s+u^2$ is a non-square in $\F$. Therefore, if $-3\in \triangle$, then $D_s$ is a nonzero square for $(q-1)/2$ values of $s$.

If $-3\in \triangle$, then for $(q+1)/2$ values of $s$, $(-3s^2-2su+4s+u^2)\in \Box$. Hence, if $-3\in \triangle$, then $D_s$ is  a non-square for $(q-1)/2$ values of $s$.

\bigskip

Collecting all the results,
we have that if $-3\in \triangle$ and $u \neq \frac{1}{2}$, then the line $\ell$ has $1$ point from the point orbit $\mathcal{P}_2$, $(q-1)/2$ points from the point orbit $\mathcal{P}_3$ and $(q+1)/2$ points from the point orbit $\mathcal{P}_4$. In addition, if either $-3\in \Box$ or $-3\in \triangle$ and $u=\frac{1}{2}$, then the line $\ell$ has $1$ point from the point orbit $\mathcal{P}_2$, $(q-1)/2$ points from the point orbit $\mathcal{P}_3$ and $(q+1)/2$ points from the point orbit $\mathcal{P}_4$.

\bigskip

($OD_2$) Now, in order to determine the plane orbit distribution of $\ell$, consider the plane $\Pi=\langle\ell, P(t)\rangle$ where $P(t)$ is a point of $\C$.
Then $\Pi$ has equation
\[(t^2u^3-tu^3)Y_0+(3t^2u-3t^2u^2-t^3+u^3)Y_1+(t^3+3tu^2-3tu-u^3)Y_2+(t-t^2)Y_3=0.\]

If $t=0$, then $Z(u^3(Y_1-Y_2))$ is a plane containing the points $(1,0,0,0)$, $(1,1,1,1)$ and $(0,0,0,1)$ of the twisted cubic $\C$. This plane belongs to $\mathcal{H}_3$.

If $t=u$ then $\Pi$ is the osculating plane containing the point $(1,u,u^2,u^3)$ of the twisted cubic $\C$, which belongs to $\mathcal{H}_1$.

For the remaining cases, it therefore suffices to consider $t\in \F\setminus \{0,1,u\}$.
If $\Pi$ contains another point $(1,v,v^2,v^3)$ of $\C$, $t\neq v$, then
 \[(t-v)f(v)=0,\]
 where
$$f(v)=(t^2-t)v^2+(3tu+u^3-3tu^2-t^2)v+tu^3-u^3.$$



If the discriminant $D_t=(t-u)^2(t^2+(6u^2-4u^3-4u)t+u^4)$ of $f(v)$ is equal to zero then
\[t^2+(6u^2-4u^3-4u)t+u^4=0,\]
since $t\neq u$.
So $D_t=0$ occurs only for values of $t\in \{t_1,t_2\}$, where
\[t_1=2u^3-3u^2-2u(u-1)\sqrt{u^2-u+1}+2u,\;\;t_2=2u^3-3u^2+2u(u-1)\sqrt{u^2-u+1}+2u.\]
Since $u^2-u+1$ is a non-square in $\F$, this implies that the discriminant $D_t\neq 0$.

%
%
%
%
%

In order to determine the values of $t$ for which the discriminant $D_t$ is a nonzero square in $\F$, define
\[g(T,Y,Z)=T^2+(6u^2-4u^3-4u)TZ+u^4Z^2-Y^2.\]
Then $\cZ(g)$ is a non-degenerate conic since $u^2-u+1$ is a non-square in $\F$. The points on the conic $\cZ(g)$ are of the following type.
A point of the form $(t,y,0)$ lies on the conic $\mathcal{Z}(g)$ if and only if $t^2-y^2=0$ giving $2$ points $(t,y,0)$ and $(t,-y,0)$. Furthermore, $\cZ(g)$ contains no points on the line $Y=0$, since $u^2-u+1$ is a non-square. A point of the form $(0,y,1)$ lies on the conic $\mathcal{Z}(g)$ if and only if $u^4=y^2$, which gives another two points $(0,u^2,1)$ and $(0,-u^2,1)$.

For $t=u$, the point $(u,y,1)$ is on the conic $\mathcal{Z}(g)$ if and only if
\[g(u,y,1)=u^2+(6u^2-4u^3-4u)u+u^4-y^2=0.\]
Since this implies $y^2=-3u^2(u-1)^2$, there two such points for $-3\in \Box$, and no such points otherwise.

For $t=1$, we obtain the points $(1,(u-1)^2,1)$ and $(1,-(u-1)^2,1)$ on the conic $\mathcal{Z}(g)$.

So the number of points of on the conic $\cZ(g)$ which we need to exclude in our count of the values of $t\in \F\setminus\{0,1,u\}$ for which $D_t$ is a nonzero square in $\F$, equals 8 points if $-3\in \Box$ and 6. The remaining points on $\cZ(g)$ of the form $(t,y,1)$ with $t\in \F\setminus\{0,1,u\}$ and $y\neq0$, come in pairs $(t,y,1)$ and $(t,-y,1)$, where each such pair counts towards one value of $t$ for which $D_t$ is a nonzero square.

A point of the form $(t,y,1)$ with $ty\neq0$ lies on the conic $\mathcal{Z}(g)$ if and only if
\[t^2+(6u^2-4u^3-4u)t+u^4=y^2.\]

We may conclude that the discriminant $D_t$ is a nonzero square in $\F$ for $(q-7)/2$
or $(q-5)/2$ values of $t\in \F \setminus \{0,1,u\}$, depending on whether $-3$ is a square in $\F$ or not.

For these values of $t\in \F$, the plane $\Pi$ contains $3$ points of the twisted cubic $\C$.
Taking into account the plane $\langle \ell, P(0)\rangle$ which also meets the twisted cubic in three points, and the fact that each such plane is counted three times, we obtain
$(q-1)/6$ planes through $\ell$ containing $3$ points of the twisted cubic $\C$
if $-3\in \Box$, and $(q+1)/6$ such planes through $\ell$ for $-3\in \triangle$.


If $D_t\in \triangle$ and $t\in \F\setminus \{0,1,u\}$, then $\Pi$ is the plane containing $1$ point $(1,t,t^2,t^3)$ of the twisted cubic $\C$. By the above the number of values of $t$ for which $\Pi$ meets $\C$ in exactly one point is therefore
 \[q-3-\frac{(q-7)}{2}=\frac{(q+1)}{2}\;\mbox{for}\;-3\in \Box,\;\; \mbox{and}\;\; q-3-\frac{(q-5)}{2}=\frac{(q-1)}{2}\;\mbox{for}\;-3\in \triangle.\]

This proves that the line $\ell$ has the plane orbit distribution $OD_2(\ell)$ as claimed.
\end{proof}

\begin{lemma}\label{lem:L_5_stab} The stabiliser in $G$ of the line $\ell=\mathcal{Z}(Y_3-3uY_2+3u^2Y_1-u^3Y_0,Y_1-Y_2)$, with $u^2-u+1$ a non-square in $\F$, has order two.
\end{lemma}
\begin{proof}
In order to simplify the computation, we map the line $\ell$ to the line $\ell'$ contained in the osculating plane $\Pi(0)$ by an element of $G$. Using the element $\varphi(\beta)$ of $G$, where $\beta$ is the element of $\mathrm{PGL}(2,q)$ induced by the matrix \[\left(
                                                                             \begin{array}{cc}
                                                                               1 & -u \\
                                                                               0 & 1 \\
                                                                             \end{array}
                                                                           \right).\]
we obtain $\ell'=\mathcal{Z}(Y_2-(1-2u)Y_1-u(1-u)Y_0,Y_3)$. Since $\ell$ (and hence $\ell'$) is contained in a unique osculating plane, the stabiliser of $\ell'$ is contained in the stabiliser of $\Pi(0)$, which has order $q(q-1)$ and consists of elements $\varphi_{c,d}\in G$, induced by the matrices of the form
\[
\left(\begin{array}{cc}
1 & 0 \\
c & d \\
\end{array} \right)
\]
where $c,d \in \F$ and $d\neq 0$.

First suppose that $u\neq 1/2$.
The condition that the image under an element $\varphi_{c,d}\in G_{\Pi(0)}$ of the point $(2u-1,u(1-u),0,0)\in \ell'$ belongs to $\ell'$ gives
\begin{eqnarray}\label{eqn:c}
c=\frac{(1-d)(1-2u)}{3u(1-u)}.
\end{eqnarray}
After substituting this value for $c$ in $\varphi_{c,d}$, the condition that the image of the point $(0,1,1-2u,0)$ under $\varphi_{c,d}$ belongs to $\ell'$ gives
\[\frac{(d^2-1)(2u-1)}{u(u-1)}=0.\]
Since $u\neq 1/2$ this implies $d^2=1$. Hence, the stabiliser in $G$ of $\ell'$ has order two.

If $u=1/2$ then the line $\ell'$ is spanned by the points $(0,1,0,0)$ and $(1,u,u(1-u),0)$. In this case, the condition that the image of $(0,1,0,0)$ under $\varphi_{c,d}$ lies on $\ell'$ gives $c=0$. Computing the image under $\varphi_{0,d}$ of $(1,u,u(1-u),0)$ then gives $(1,0,d^2u(1-u),0)$ which belongs to $\ell'$ if and only if $d^2=1$. Again the stabiliser in $G$ of $\ell'$ has order two.
\end{proof}

It follows from Lemma \ref{lem:L_5_stab} that there are $q(q-1)/2$ lines in the unique osculating plane through $\ell$ which belong to the $G$-orbit of $\ell$
\begin{lemma}\label{lem:L_5_one_orbit}
The set of lines in osculating planes which are not contained in the union of the $G$-orbits $\cL_1\cup \cL_2\cup \cL_3\cup \cL_4$ form one $G$-orbit, which we denote by $\cL_5$.
\end{lemma}
\begin{proof}
Consider an osculating plane $\Pi(t)$ at the point $P(t)$ of $\C$. The number of lines of the $G$-orbit $\cL_1$ in $\Pi(t)$ is $q$, since each such line is the intersection of $\Pi(t)$ with exactly one other osculating plane $\Pi(s)$ with $s\neq t$. Furthermore, there is a unique tangent line in $\Pi(t)$, and there are $q$ lines in $\Pi(t)$ which belong to the $G$-orbit $\cL_3$ (non-tangent unisecants).
By Remark \ref{rem:size_L_4}, this leaves $q(q-1)/2$ lines in $\Pi(t)$ which not are contained in the union of the $G$-orbits $\cL_1\cup \cL_2\cup \cL_3\cup \cL_4$. By Lemma \ref{lem:L_5_stab} these lines form one $G$-orbit, since the stabiliser in $G$ of an osculating plane has size $q(q-1)$.
\end{proof}


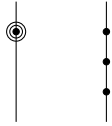
\begin{figure}[!h]
\centering
\begin{tikzpicture}[scale=0.4]
\draw (-3,0)--(-3,4);
\draw[fill=black] (-3,3) circle (3.0 pt);
\draw (-3,3) circle (6.0 pt);
\draw (-3,3) circle (9.0 pt);
\draw (0,0)--(0,4);
\draw[fill=black] (0,3) circle (3.0 pt);
\draw[fill=black] (0,2) circle (3.0 pt);
\draw[fill=black] (0,1) circle (3.0 pt);
\draw (0,-1) node {};
\end{tikzpicture}
\vspace{-1em}
\caption{Diagram of the orbit $\cL_5$}
\end{figure}

 \begin{corollary} There is no line with plane orbit distribution $OD_2(\ell)=[1,0,0,d_2,e_2]$ where $d_2\geq1$.
 \end{corollary}
\begin{proof}
This simply follows by counting the lines in the orbits $\cL_i$, $i\in \{1,\ldots,5\}$, which are contained in a fixed osculating plane.
\end{proof}

\section{Lines meeting the twisted cubic}

Since the polar of a line meeting the twisted cubic is a line contained in an osculating plane of the twisted cubic, the orbit distributions determined in the previous sections imply the orbit distributions of all lines meeting the twisted cubic. In this section we collect the results for the $G$-orbits which consist of lines meeting the twisted cubic, but not contained in an osculating plane.

The polar of a line belonging to the $G$-orbit $\cL_1$ is a chord of the twisted cubic $\C$, and so the orbit distributions follow from the orbit distributions of $\cL_1$.

\begin{lemma} There is one orbit $\cL_6=\cL_1^\perp=\mathcal{O}_1$ of real chords of the twisted cubic $\C$ in $\PG$. A real chord $\ell$ of the twisted cubic $\C$ has plane orbit distribution $OD_2(\ell)=[0,2,(q-1),0,0]$ and point orbit distribution $OD_0(\ell)=[2,0,0,(q-1),0]$ for $q\equiv 5$ (mod $6$). If $q\equiv 1$ (mod $6$) then the orbit distributions are
$OD_2(\ell)=[0,2,(q-1),0,0]$ and $OD_0(\ell)=[2,0,\frac{(q-1)}{3},0,\frac{2(q-1)}{3}]$.
\end{lemma}

\begin{figure}[!h]
\centering
\begin{tikzpicture}[scale=0.4]
\draw (-3,0)--(-3,4);
\draw[fill=black] (-3,3) circle (3.0 pt);
\draw (-3,3) circle (6.0 pt);
\draw[fill=black] (-3,1) circle (3.0 pt);
\draw (0,0)--(0,4);
\draw[fill=black] (0,3) circle (3.0 pt);
\draw[fill=black] (0,1) circle (3.0 pt);
\draw (0,1) circle (6.0 pt);
\draw [dashed] (-2.7,3)--(-0.3,3);
\draw [dashed] (-2.7,1)--(-0.3,1);
\end{tikzpicture}
\caption{Diagram of the pencil of cubics corresponding to $\cL_6$}
\end{figure}
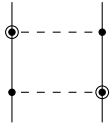

Now, let $\ell$ be an external line contained in the osculating plane $\Pi$ at the point $P$ of $\C$. Since self polar lines through $P$ are all lines through $P$ which are contained in $\Pi$, the line $\ell^{\sigma}$ is a unisecant through $P$ which is not contained in $\Pi$. Since we already determined the point orbit distribution and the line orbit distribution of the external lines in osculating planes of $\C$ (the $G$-orbits $\cL_4$ and $\cL_5$) we immediately have the following lemma's.

\begin{lemma} There is one orbit $\cL_7=\cL_4^\perp\subset\mathcal{O}_5$ of unisecants not in osculating planes of the twisted cubic $\C$ in $\PG$ with plane orbit distribution $OD_2(\ell)=[0,3,\frac{(q-3)}{2},\frac{(q-1)}{2},0]$ and point orbit distribution $OD_0(\ell)=[1,2,\frac{(q-5)}{6},\frac{(q-3)}{2},\frac{(q+1)}{3}]$  for $-3 \in \triangle$ and $OD_2(\ell)=[0,3,\frac{(q-3)}{2},\frac{(q-1)}{2},0]$, $OD_0(\ell)=[1,2,\frac{(q-7)}{6},\frac{(q-1)}{2},\frac{(q-1)}{3}]$  for $-3 \in \Box$.
\end{lemma}

\begin{figure}[!h]
\centering
\begin{tikzpicture}[scale=0.4]
\draw (-3,0)--(-3,4);
\draw[fill=black] (-3,3) circle (3.0 pt);
\draw (-3,3) circle (6.0 pt);
\draw[fill=black] (-3,1) circle (3.0 pt);
\draw (0,0)--(0,4);
\draw[fill=black] (0,3) circle (3.0 pt);
\draw[fill=black] (0,1) circle (3.0 pt);
\draw (0,1) circle (6.0 pt);
\draw [dashed] (-2.7,3)--(-0.3,3);
\end{tikzpicture}
\caption{Diagram of the pencil of cubics corresponding to $\cL_7$}
\end{figure}
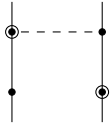

\begin{lemma} There is one orbit $\cL_8=\cL_5^\perp\subset\mathcal{O}_5$ of unisecants not in osculating planes of the twisted cubic $\C$ in $\PG$ with plane orbit distribution $OD_2(\ell)=[0,1,\frac{(q-1)}{2},\frac{(q+1)}{2},0]$ and point orbit distribution $OD_0(\ell)=[1,0,\frac{(q+1)}{6},\frac{(q-1)}{2},\frac{(q+1)}{3}]$ for $-3 \in \triangle$, $OD_2(\ell)=[0,1,\frac{(q-1)}{2},\frac{(q+1)}{2},0]$ and $OD_0(\ell)=[1,0,\frac{(q-1)}{6},\frac{(q+1)}{2},\frac{(q-1)}{3}]$ for $-3 \in \Box$.
\end{lemma}

\begin{figure}[!h]
\centering
\begin{tikzpicture}[scale=0.4]
\draw (-3,0)--(-3,4);
\draw[fill=black] (-3,3) circle (3.0 pt);
\draw (-3,3) circle (6.0 pt);
\draw[fill=black] (-3,1) circle (3.0 pt);
\draw (0,0)--(0,4);
\draw[fill=black] (0,3) circle (3.0 pt);
\draw[fill=black] (0,2) circle (3.0 pt);
\draw[fill=black] (0,1) circle (3.0 pt);
\draw [dashed] (-2.7,3)--(-0.3,3);
\end{tikzpicture}
\caption{Diagram of the pencil of cubics corresponding to $\cL_8$}
\end{figure}
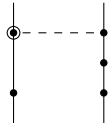

\begin{lemma} If a pencil of cubics on ${\mathrm{PG}}(1,q)$  has a base point, then it belongs to one of the orbits $\cL_2$, $\cL_3$, $\cL_6$, $\cL_7$ or $\cL_8$.
\end{lemma}
\begin{proof}
The line corresponding to a pencil of cubics on ${\mathrm{PG}}(1,q)$ with a base point $P$ corresponds to a line of $\PG$ meeting the twisted cubic $\C$ in the point $\nu_3(P)$, and these lines are contained in the union of the orbits $\cL_2$, $\cL_3$, $\cL_6$, $\cL_7$, and $\cL_8$.
\end{proof}

\section{Orbits of imaginary chords and imaginary axes}

We start with the imaginary chords.

\begin{lemma} There is one orbit $\cL_9=\mathcal{O}_3$ of imaginary chords of the twisted cubic $\C$ in $\PG$ with plane orbit distribution $[0,0,0,q+1,0]$. The point orbit distribution of an imaginary chords $\ell$ is
$$
OD_0(\ell)=[0,0,\frac{(q+1)}{3},0,\frac{2(q+1)}{3}].
$$
if $-3$ is a non-square in $\bF_q$ and
$$
OD_0(\ell)=[0,0,0,q+1,0],
$$
if $-3$ is a square in $\bF_q$.
\end{lemma}
\begin{proof}
Consider an imaginary chord $\ell$. Let $P$, $P^\tau$ denote the two points of $\ell(q^2)\cap \C(q^2)$ (here $\tau$ denotes the Frobenius collineation of ${\mathrm{PG}}(3,q^2)$). Since a plane $\Pi=\langle\ell, P(t)\rangle$ with $P(t)\in \C$ has the property that $\Pi(q^2)$ meets $\C(q^2)$ in the three points $P$, $P^\tau$, and $P(t)$, the plane $\Pi$ belongs to $\cH_4$. This proves that $OD_2(\ell)=[0,0,0,q+1,0]$, as claimed.

\bigskip

Consider two imaginary chords $\ell$ and $\ell'$. Then $\ell(q^2)$ and $\ell'(q^2)$ belong to the same orbit under the extension $G(q^2)\cong \PGL(2,q^2)$ of $G$. Let $\alpha\in G(q^2)$ be such that $\ell(q^2)^\alpha=\ell'(q^2)$.
Consider three distinct planes $\Pi_1$, $\Pi_2$, and $\Pi_3$ through $\ell$ in $\PG$. By the above each of these planes meets $\C$ in a points. Let $P_1$, $P_2$, and $P_3$ denote these points of $\C$. Then $\Gamma(P,P^\tau,P_1,P_2,P_3)$ is a frame of ${\mathrm{PG}}(3,q^2)$, since $\Gamma$ consists of five points on the twisted cubic $\C(q^2)$.
But then $\alpha\tau\alpha^{-1}\tau=id$, since it fixes frame $\Gamma$. Therefore $\alpha \tau=\tau \alpha$, which implies $\alpha \in \PGL(4,q)$. Hence $\ell$ and $\ell'$ belong to the same $G$-orbit.

\bigskip

To determine the point orbit distribution, consider the imaginary chord
 \[\ell=\mathcal{Z}(Y_0+bY_2,Y_1+bY_3),\]
whose points can be parameterized as
\[\{(-b,-bs,1,s)\;\;|\;\;s\in\F\}\cup\{(0,-b,0,1)\},\]
where $-b$ a non-square in $\bF_q$.

First note that the point $(0,-b,0,1) \in \ell$ lies on the osculating plane $\Pi(\infty)$ and lies on the osculating plane $\Pi(t)$ if and only if $-3t^2b+1=0$.
This means that if $-3$ is a non-square in $\bF_q$ then  $(0,-b,0,1)$ is on three osculating planes, otherwise it is only on one osculating plane.

A point $(-b,-bs,1,s)\in \ell$ lies in the osculating plane $\Pi(t)$ of the twisted cubic $\C$ if
\[ bt^3-3t^2bs-3t+s=0.\]
In order to solve this cubic equation we make the substitution $t\mapsto t+s$. This gives
\[f(t)=t^3+(\frac{-3}{b})(bs^2+1)t+(\frac{-2s}{b})(bs^2+1).\]
whose discriminant is
\[D_f=108\frac{(bs^2+1)^2}{b^3},\]
which is not zero because $-b$ is a non-square in $\bF_q$. This implies that $OD_0(\ell)=[0,0,a_3,a_4,a_5]$.

\bigskip

Also $D_f$ is a non-square if and only if $-3$ is a nonzero square in $\bF_q$. In this case every point of $\ell$ lies on exactly one osculating plane defined over $\bF_q$ (and two conjugate osculating planes of $\C(q^2)$). This gives $OD_0(\ell)=[0,0,0,q+1,0]$.

\bigskip

If -3 is a non-square in $\bF_q$, then $D_f$ is a square in $\bF_q\setminus \{0\}$. In this case $OD_0(\ell)=[0,0,a_3,0,a_5]$, and each point of $\ell$ either lies on three distinct osculating planes of $\C$ or on zero osculating planes of $\C$. Since there are $q+1$ osculating planes, each of them intersecting $\ell$, we obtain
$$
OD_0(\ell)=[0,0,\frac{(q+1)}{3},0,\frac{2(q+1)}{3}].
$$
This concludes the proof.
\end{proof}

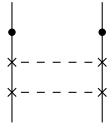
\begin{figure}[!h]
\centering
\begin{tikzpicture}[scale=0.4]
\draw (-3,0)--(-3,4);
\draw[fill=black] (-3,3) circle (3.0 pt);
\draw[fill=black] (-3,2) node[cross=2.0 pt]{};
\draw[fill=black] (-3,1) node[cross=2.0 pt]{};

\draw (0,0)--(0,4);
\draw[fill=black] (0,3) circle (3.0 pt);
\draw[fill=black] (0,2) node[cross=2.0 pt]{};
\draw[fill=black] (0,1) node[cross=2.0 pt]{};
\draw [dashed] (-2.7,2)--(-0.3,2);
\draw [dashed] (-2.7,1)--(-0.3,1);
\end{tikzpicture}
\caption{Diagram of the pencil of cubics corresponding to $\cL_9$}
\end{figure}

\par Since the polar line of an imaginary line is an imaginary axes of $\C$, we have the following lemma.
\begin{lemma} There is one orbit $\cL_{10}=\cL_9^\perp=\mathcal{O}^\perp_3$ of imaginary axes of $\C$ in $\PG$ with plane orbit distribution $OD_2(\ell)=OD_0(\ell^\perp)$ and
$OD_0(\ell)=OD_2(\ell^\perp)$, for $\ell\in \cL_{10}$.
\end{lemma}

\begin{figure}[!h]
\centering
\begin{tikzpicture}[scale=0.4]
\draw (-3,0)--(-3,4);
\draw[fill=black] (-3,3) circle (3.0 pt);
\draw[fill=black] (-3,2) node[cross=2.0 pt]{};
\draw[fill=black] (-3,1) node[cross=2.0 pt]{};

\draw (0,0)--(0,4);
\draw[fill=black] (0,3) circle (3.0 pt);
\draw[fill=black] (0,2) node[cross=2.0 pt]{};
\draw[fill=black] (0,1) node[cross=2.0 pt]{};
\end{tikzpicture}
\caption{Diagram of the pencil of cubics corresponding to $\cL_{10}$}
\end{figure}
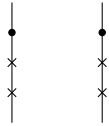

\vspace{-1em}
\section{Conclusions}
In this paper we determined the point orbit distributions and plane orbit distributions of ten orbits of lines in $\PG$ with respect to the twisted cubic, for $q$ odd and not divisible by three. Our results are summarised in following table.

\begin{center}
\begin{tabular}{|c|c|c|c|}
\hline
 orbit  &  $OD_2$ & $OD_0$ &  base\\
\hline
  $\cL_1=\mathcal{O}^\perp_1$ &  $[2, 0, 0, (q-1), 0]$ & $[0, 2, (q-1), 0, 0]$ &$\emptyset$\\
   $-3\in \Box$  &  $[2, 0, \frac{(q-1)}{3},0,\frac{2(q-1)}{3}]$ & $[0, 2, (q-1),0, 0]$  &$\emptyset$ \\
\hline
  $\cL_2=\mathcal{O}_2$ &  $[1, q, 0, 0, 0]$ & $[1, q, 0, 0, 0]$  &$\mathcal{Z}(X_1^2)$\\
\hline
 $\cL_3=\mathcal{O}_4$ &  $[1, 1, \frac{(q-1)}{2}, \frac{(q-1)}{2}, 0]$& $[1, 1, \frac{(q-1)}{2},  \frac{(q-1)}{2}, 0]$ &$\mathcal{Z}(X_1)$\\
\hline
 $\cL_4\subset\mathcal{O}^\perp_5$ &  $[1, 2, \frac{(q-5)}{6} , \frac{(q-3)}{2}, \frac{(q+1)}{3}] $, &$[0, 3, \frac{(q-3)}{2}, \frac{(q-1)}{2}, 0] $&  $\emptyset$\\
$-3\in \Box$ &  $[1, 2, \frac{(q-7)}{6} , \frac{(q-1)}{2}, \frac{(q-1)}{3}] $, &$[0, 3, \frac{(q-3)}{2},\frac{(q-1)}{2}, 0] $&  $\emptyset$\\
\hline
$\cL_5\subset\mathcal{O}^\perp_5$ &  $[1, 0, \frac{(q+1)}{6}, \frac{(q-1)}{2}, \frac{(q+1)}{3}] $ &$[0, 1, \frac{q-1}{2}, \frac{q+1}{2}, 0]$ ,  &$\emptyset$\\
$-3\in \Box$ &  $[1, 0, \frac{(q-1)}{6}, \frac{(q+1)}{2}, \frac{(q-1)}{3}] $&$[0, 1, \frac{q-1}{2}, \frac{q+1}{2}, 0]$  &$\emptyset$\\
\hline
$\cL_6=\mathcal{O}_1$ &  $[0,2,(q-1),0,0]$&$[2, 0, (q-1),0 , 0]$&  $\mathcal{Z}(X_0,X_1)$\\
  $-3\in \Box$  &  $[0,2,(q-1),0,0]$&$[2, 0, \frac{(q-1)}{3},0,\frac{2(q-1)}{3}]$ &$\mathcal{Z}(X_0,X_1)$\\
\hline
$\cL_7\subset \mathcal{O}_5$ &  $[0,3,\frac{(q-3)}{2},\frac{(q-1)}{2},0]$&$[1, 2, \frac{(q-5)}{6},\frac{(q-3)}{2},  \frac{(q+1)}{3}]$ &$\mathcal{Z}(X_1)$\\
 $-3\in \Box$ &  $[0,3,\frac{(q-3)}{2},\frac{(q-1)}{2},0]$&$[1, 2, \frac{(q-7)}{6}, \frac{(q-1)}{2}, \frac{(q-1)}{3}] $ &$\mathcal{Z}(X_1)$\\
\hline
$\cL_8\subset \mathcal{O}_5$ &  $[0,1,\frac{(q-1)}{2},\frac{(q+1)}{2},0]$&$[1, 0,\frac{(q+1)}{6}, \frac{(q-1)}{2}, \frac{(q+1)}{3}]$ &  $\mathcal{Z}(X_1)$\\
$-3\in \Box$ & $[0, 1, \frac{q-1}{2}, \frac{q+1}{2}, 0]$ &$[1, 0, \frac{(q-1)}{6}, \frac{(q+1)}{2}, \frac{(q-1)}{3}] $  &$\mathcal{Z}(X_1)$\\
\hline
$\cL_9=\mathcal{O}_3$ & $[0,0,0,(q+1),0]$ & $[0,0,\frac{(q+1)}{3},0,\frac{2(q+1)}{3}]$  &$\mathcal{Z}(X_0^2+bX_1^2)$\\
$-3\in \Box$ & $[0,0,0,(q+1),0]$ &$[0,0,0,(q+1),0]$  &$\mathcal{Z}(X_0^2+bX_1^2)$\\
\hline
$\cL_{10}=\mathcal{O}^\perp_3$ & $[0,0,\frac{(q+1)}{3},0,\frac{2(q+1)}{3}]$ & $[0,0,0,(q+1),0]$  &$\emptyset$\\
$-3\in \Box$ & $[0,0,0,(q+1),0]$ & $[0,0,0,(q+1),0]$  &$\emptyset$\\
\hline
\end{tabular}
\end{center}

 The remaining $G$-orbits of lines partition the class $\cO_6$ consisting of external lines to $\C$ which are not imaginary chords and which are not contained in osculating planes. There are in total $q(q-1)(q^2-1)$ such lines. As far as we know the $G$-orbits of these lines are not classified yet. We conjecture that there are $2q-2$ or $2q-4$ $G$-orbits contained in $\cO_6$ depending on whether $q$ is 1 or 5 modulo 6. For reasons of clarity, we state our conjecture in terms of the total number of $G$-orbits of lines in $\PG$, i.e. including the 10 $G$-orbits which we already know.

\begin{conjecture}
The number of line orbits under the action of $G$ is $9+(2q-1)$, if $q\equiv 1$ (mod $6$), $G$ is $9+(2q-3)$ if $q\equiv 5$ (mod $6$).
\end{conjecture}

\begin{remark}
The results reported in this paper are part of the PhD thesis of the first author; they were described in an online talk given by the second author \cite{talk} as part of the seminar series
{\it Seminars and e-Seminars Arbeitsgemeinschaft in Codierungstheorie und Kryptographie}, at the Mathematics Institute of the University of Zurich in Switzerland; and they are in accordance with the results contained in \cite{tcubicorbits}, and \cite{BlPeSz} recently submitted to the arXiv.
\end{remark}

\bibliographystyle{amsplain}

\end{document}